\documentclass[11pt,twoside,reqno]{amsart}

\usepackage{amssymb, latexsym}
\usepackage[italian,english]{babel}
\usepackage{amsmath,amsfonts,amsthm}
\usepackage{paralist}
\usepackage[top=2.50cm, bottom=2.50cm, left=2.50cm, right=2.50cm]{geometry}

\usepackage{color}

\numberwithin{equation}{section}

\usepackage{xspace}
\usepackage[bookmarksnumbered,colorlinks]{hyperref}
\usepackage{graphics}
\def\bb#1\eb{\textcolor{blue}
{#1}} %
\def\br#1\er{\textcolor{red}
{#1}} %
\def\bv#1\ev{\textcolor{green}
{#1}} %
\def\bc#1\ec{\textcolor{cyan}
{#1}} %

\usepackage{graphics}

\def\Xint#1{\mathchoice
  {\XXint\displaystyle\textstyle{#1}}%
  {\XXint\textstyle\scriptstyle{#1}}%
  {\XXint\scriptstyle\scriptscriptstyle{#1}}%
  {\XXint\scriptscriptstyle\scriptscriptstyle{#1}}%
  \!\int}
\def\XXint#1#2#3{{\setbox0=\hbox{$#1{#2#3}{\int}$}
  \vcenter{\hbox{$#2#3$}}\kern-.5\wd0}}
\def\-int{\Xint -}

\newcommand{\R}{\mathbb{R}}
\newcommand{\N}{\mathbb{N}}

\newcommand{\h}{\textit{H}^{s}(\R^{N})}
\newcommand{\X}{\textit{X}^{s}(\R^{N+1}_{+})}
\newcommand{\Y}{\textit{X}^{s}_{r}(\R^{N+1}_{+})}

\DeclareMathOperator{\dive}{div}
\DeclareMathOperator{\supp}{supp}

\DeclareMathOperator{\I}{\mathcal{I}}
\DeclareMathOperator{\e}{\varepsilon}

\newtheorem{lem}{Lemma}[section]
\newtheorem{thm}{Theorem}[section]
\newtheorem{defn}{Definition}[section]

\newtheorem{ex}{Example}[section]
\newtheorem{remark}{Remark}[section]

\begin{document}
\title[multiple solutions]{Multiple solutions for a class of nonhomogeneous fractional Schr\"odinger equations in $\R^{N}$} 

\author[V. Ambrosio]{Vincenzo Ambrosio}
\address{Dipartimento di Scienze Pure e Applicate (DiSPeA),
Universit\`a degli Studi di Urbino 'Carlo Bo'
Piazza della Repubblica, 13
61029 Urbino (Pesaro e Urbino, Italy)}
\email{vincenzo.ambrosio@uniurb.it}
\author[H. Hajaiej]{Hichem Hajaiej}

\address{Department of Mathematics
California State University, Los Angeles                            
5151 State University Drive Los Angeles, CA 90032}
\email{hichem.hajaiej@gmail.com}

\keywords{Fractional Laplacian; mountain pass theorem; extension method; positive solutions}
\subjclass[2010]{35A15, 35J60, 35R11, 45G05}

\begin{abstract}
This paper is concerned with the following fractional Schr\"odinger equation
\begin{equation*}
\left\{
\begin{array}{ll}
(-\Delta)^{s} u+u= k(x)f(u)+h(x) \mbox{ in } \R^{N}\\
u\in H^{s}(\R^{N}), \, u>0 \mbox{ in } \R^{N},
\end{array}
\right.
\end{equation*}
where $s\in (0,1)$, $N> 2s$, $(-\Delta)^{s}$ is the fractional Laplacian,  $k$ is a bounded positive function, $h\in L^{2}(\R^{N})$, $h\not \equiv 0$ is nonnegative and  $f$ is either asymptotically linear or superlinear at infinity.\\
By using the $s$-harmonic extension technique and suitable variational methods, we prove the existence of at least two positive solutions for the problem under consideration, provided that $|h|_{2}$ is sufficiently small.  
\end{abstract}

\maketitle
\section{Introduction}

\noindent
In this paper we deal with the existence of positive solutions for the following nonlinear fractional equation 
\begin{equation}\label{P}
\left\{
\begin{array}{ll}
(-\Delta)^{s} u+u= k(x)f(u)+h(x) \mbox{ in } \R^{N} \\
u\in H^{s}(\R^{N}), \, u>0 \mbox{ in } \R^{N},
\end{array}
\right.
\end{equation}
with $s\in (0,1)$, $N> 2s$, $k$ is a bounded positive function, $h\in L^{2}(\R^{N})$, $h\geq 0$, $h\not\equiv 0$, and the nonlinearity $f: \R\rightarrow \R$
is a smooth function which can be either asymptotically linear or superlinear at infinity.
Here, $(-\Delta)^{s}$ is the fractional Laplacian which can be defined for $u: \R^{N}\rightarrow \R$ belonging to the Schwartz space $\mathcal{S}(\R^{N})$ of rapidly decaying $C^{\infty}$-functions in $\R^{N}$, by setting 
$$
(-\Delta)^{s}u(x)=-\frac{C(N,s)}{2} \int_{\R^{N}} \frac{u(x+y)+u(x-y)-2u(x)}{|y|^{N+2s}} dy  \quad (x\in \R^{N})
$$
where
$$
C(N, s)=\left(\int_{\R^{N}} \frac{1-\cos (x_{1})}{|x|^{N+2s}} \, dx \right)^{-1};
$$
see for instance \cite{DPV}.\\
A basic motivation to study (\ref{P}), comes from the nonlinear fractional Schr\"odinger equation 
\begin{equation}\label{FSE}
\imath \frac{\partial \psi}{\partial t}=(-\Delta)^{s} \psi+V(x)\psi -g(x, |\psi|) \quad (t, x)\in \R\times \R^{N},
\end{equation}
when we are interested in standing wave solutions, namely solutions of the form $\psi(t, x)=u(x)e^{-\imath ct}$. 
Indeed, it is easily observed that a function $\psi(t, x)$ of this form satisfies (\ref{FSE}) if and only if $u$ is a solution of 
\begin{equation}\label{AFSE}
(-\Delta)^{s}u+V(x)u=g(x, u) \mbox{ in } \R^{N}
\end{equation}
with $V(x)=1+c$ and $g(x, u)=k(x)f(u)+h(x)$.\\
The equation (\ref{FSE}) has been proposed by Laskin in \cite{Laskin1, Laskin2} and it is fundamental in quantum mechanics, because it appears in problems involving nonlinear optics, plasma physics and condensed matter physics. \\
When $s=1$, the equation (\ref{AFSE}) becomes the classical Schr\"odinger equation
\begin{equation*}
-\Delta u+V(x)u=g(x, u) \mbox{ in } \R^{N},
\end{equation*}
which has been widely investigated by many authors, particularly on the existence of ground state solutions, positive solutions, sign-changing solutions and multiplicity of standing wave solutions. 
Since we cannot review the huge bibliography here, we just mention the works \cite{BLW, BW, BL1, FW, Rab}.
\noindent

Recently, a considerable attention has been focused on the study of problems involving fractional and non-local operators. This interest is motivated by both the theoretical research and the large number of applications such as phase transitions, flames propagation, chemical reaction in liquids, population dynamics, American options in finance, crystal dislocation, obstacle problems, soft thin films, semipermeable membranes, conservation laws, quasi-geostrophic flows, and so on. \\
Since the current literature is really too wide to attempt any reasonable account here, we derive the interested reader to \cite{DPV, MBRS}, where an elementary introduction to this topic is given. 
 
Concerning the fractional Schr\"odinger equation (\ref{AFSE}), in the last decade, many several existence and multiplicity results have been obtained by using different variational methods.\\
Felmer et al. \cite{FQT} dealt with the existence and the symmetry of positive solutions for (\ref{AFSE}) with $V(x)=1$ and involving a superlinear nonlinearity $g(x, u)$ satisfying the Ambrosetti-Rabinowitz condition.
Secchi \cite{Secchi1} obtained via minimization on Nehari manifold, the existence of ground state solutions to (\ref{AFSE}) when the nonlinearity is superlinear and subcritical, and the potential $V(x)$ satisfies suitable assumptions as $|x|\rightarrow \infty$.
Chang and Wang \cite{CW} (see also \cite{A2}) investigated the existence of nontrivial solutions  to (\ref{AFSE}) with $V(x)=1$, $g$ is autonomous  and verifies Berestycki-Lions type assumptions.
Frank et al. \cite{FLS} showed uniqueness and nondegeneracy of ground state solutions for (\ref{AFSE}) with $V(x)=1$ and $g(x, u)=|u|^{q}u$ for all admissible exponents $q$.
Cho et al. \cite{H3} obtained the existence of standing waves using the method of concentration-compactness by studying the associated constrained minimization problem, and showed the orbital stability of standing waves which are the minimizers of the associate variational problem. 
Pucci et al. \cite{P} established via Mountain Pass Theorem and Ekeland variational principle, the existence of multiple solutions for a Kirchhoff fractional Schr\"odinger equation involving a nonlinearity satisfying the Ambrosetti-Rabinowitz condition, a positive potential $V(x)$ verifying suitable assumptions, and in presence of a perturbation term.
Figueiredo and Siciliano \cite{FS} established a multiplicity result for a fractional Schr\"odinger equation via Ljusternick- Schnirelmann and Morse theory. We also mention the papers \cite{A1, A3, A4, Chang1, H3, FMV, H1, H2, H4, MBR, Secchi3, SZ} where further results related to (\ref{AFSE}) are given.
\smallskip

\noindent
After an accurate bibliographic review, we have realised that there are only few papers concerning the existence and the multiplicity of solutions for nonhomogeneous problems in non-local setting \cite{Colorado, Servadei, Torres}, and this is 
surprising, because, in the classic framework, such type of problems have been extensively investigated by many authors \cite{BB, CZ, Jtwo, Struwe, Zhu, ZZ}.

Strongly motivated by this fact, the purpose of this work is to investigate the existence and the multiplicity of positive solutions for the nonhomogeneous equation (\ref{P}), under the effect of a small perturbation $h\in L^{2}(\R^{N})$, and requiring suitable assumptions on the nonlinearity $f$.\\
More precisely, along the paper, we assume that $f$ satisfies the following conditions:
\begin{compactenum}[($f1$)]
\item $f\in C^{1}(\R, \R^{+})$, $f(0)=0$ and $f(t)=0$ for $t\leq 0$;
\item $\displaystyle{\lim_{t\rightarrow 0^{+}} \frac{f(t)}{t}=0}$;
\item there exists $p\in (1, \frac{N+2s}{N-2s})$ such that $\displaystyle{\lim_{t\rightarrow +\infty} \frac{f(t)}{t^{p}}=0}$;
\item there exists $l\in (0, \infty]$ such that $\displaystyle{\lim_{t\rightarrow +\infty} \frac{f(t)}{t}=l}$.
\end{compactenum}
\noindent
Let us note that $(f1)$-$(f3)$, yield that for any $\varepsilon>0$ there exists $C_{\varepsilon}>0$ such that 
\begin{equation}\label{2.7}
|F(t)|\leq \frac{\varepsilon}{2}t^{2}+\frac{1}{p+1}C_{\varepsilon}|t|^{p+1} \mbox{ for all } t\in \R,
\end{equation}
while $(f4)$ implies that $f$ is asymptotically linear if $l<\infty$, or superlinear  when $l=\infty$. \\
Then, under the above assumptions, we are able to prove the existence of weak solutions to (\ref{P}).\\
We recall that the precise meaning of solution to (\ref{P}) is the following: 
\begin{defn}
We say that  $u\in H^{s}(\R^{N})$ is a weak solution to $(\ref{P})$ if verifies the following condition
$$
\iint_{\R^{2N}} \frac{(u(x)-u(y))}{|x-y|^{N+2s}} (\varphi(x)-\varphi(y))\, dx dy+\int_{\R^{N}} u \varphi \, dx=\int_{\R^{N}} [k(x) f(u)+h(x)] \varphi \, dx
$$
for any $\varphi\in \h$.
\end{defn}
Here
$$
H^{s}(\R^{N})= \left\{u\in L^{2}(\R^{N}) : \frac{|u(x)-u(y)|}{|x-y|^{\frac{N+2s}{2}}} \in L^{2}(\R^{2N}) \right \}.
$$
is the fractional Sobolev space \cite{DPV, MBRS}; see Section $2$ below.

Due to the presence of the fractional Laplacian, which is a nonlocal operator, we prefer to analyze (\ref{P}) by using the $s$-harmonic extension method \cite{CS}.
This procedure is commonly used to study fractional problems, since it allows us to write a given nonlocal equation in a local way and to adapt known  techniques of the Calculus of Variations to these kind of problems.  \\
Hence, instead of (\ref{P}), we can consider the following degenerate elliptic equation with a nonlinear Neumann boundary condition 
\begin{equation}\label{R0}
\left\{
\begin{array}{ll}
\dive(y^{1-2s} \nabla U)=0 & \mbox{ in } \R^{N+1}_{+}=\{(x, y)\in \R^{N}\times \R: y>0 \} \\
\frac{\partial U}{\partial \nu^{1-2s}}=\kappa_{s}[-U+k(x) f(U)+h(x)]  & \mbox{ on } \partial\R^{N+1}_{+}=\R^{N}\times \{0\} \\
\end{array},
\right.
\end{equation}
where $\kappa_{s}$ is a suitable constant; see \cite{CS}.
Taking into account  this fact, we are able to resemble some variational techniques developed in the papers \cite{JT, LWZ, StZ, WZ}, dealing with asymptotically or superlinear classical problems, by introducing the following functional 
$$
I(U)=\frac{1}{2}\Bigl[\kappa^{-1}_{s}\iint_{\R^{N+1}_{+}} y^{1-2s} |\nabla U|^{2} \, dx dy+\int_{\R^{N}} U(x, 0)^{2} \, dx \Bigr]-\int_{\R^{N}} k(x) F(U(x, 0)) \, dx- \int_{\R^{N}} h(x) U(x, 0) \, dx
$$
defined on the weight Sobolev space $\X$ consisting of functions $U:\R^{N+1}_{+}\rightarrow \R$ such that 
$$
\iint_{\R^{N+1}_{+}} y^{1-2s} |\nabla U|^{2} \, dx dy+\int_{\R^{N}} U(x, 0)^{2} \, dx<\infty.
$$
Clearly, this functional simplification creates some additional difficulties to overcome such as, for instance, some weighted embedding are needed (see Theorem \ref{DMV}) to obtain some convergence results (compare with Lemma \ref{lem2.1}).
Moreover, the arguments used in \cite{JT, LWZ} to prove the non-existence of solutions for certain eigenvalues problems, have to be handled carefully in order to take care the trace of the involved functions (see Lemma \ref{lem2.5}).

\noindent
Now, we state our first main result concerning the existence of positive  solutions to (\ref{P}) in the asymptotically linear case, that is $l<\infty$.
\begin{thm}\label{thm1}
Let $s\in (0,1)$ and $N> 2s$. Assume that $h\in L^{2}(\R^{N})$, $h(x)\geq 0$, $h(x)\not\equiv 0$ and $k\in L^{\infty}(\R^{N}, \R_{+})$ verifies the following condition:
\begin{compactenum}[($K$)]
\item there exists $R_{0}>0$ such that 
\begin{equation}\label{1.3}
\sup \left\{\frac{f(t)}{t}: t>0\right\}<\inf\left\{\frac{1}{k(x)}: |x|\geq R_{0}\right\}.
\end{equation}
\end{compactenum}
Let us suppose that $f$ verifies $(f1)$-$(f4)$ and $\mu^{*}\in (l, \infty)$ where
\begin{equation}\label{1.4}
\mu^{*}=\inf\left\{\int_{\R^{N}} (|(-\Delta)^{\frac{s}{2}}u|^{2}+u^{2}) \,dx: u\in H^{s}(\R^{N}), \int_{\R^{N}} k(x) u^{2} \, dx=1\right\}.
\end{equation}
Let us assume that 
\begin{equation}\label{HV}
|h|_{2}<m:=\max_{t\geq 0}\left[\left(\frac{1}{2}-\frac{\varepsilon}{2}|k|_{L^{\infty}(\R^{N})}\right)t-\frac{C_{\varepsilon}}{p+1}S_{*}^{p+1}t^{p}|k|_{L^{\infty}(\R^{N})}\right],
\end{equation}
where $\varepsilon\in (0, |k|^{-1}_{L^{\infty}(\R^{N})})$ is fixed and $S_{*}$ is the best Sobolev constant of the embedding $H^{s}(\R^{N})\subset L^{2^{*}_{s}}(\R^{N})$.
Then, the problem $(\ref{P})$ possesses at least two positive solutions $u_{1}, u_{2}\in H^{s}(\R^{N})$ with the property that $E(u_{1})<0<E(u_{2})$.
Here $E: H^{s}(\R^{N})\rightarrow \R$ is the energy functional associated to $(\ref{P})$, that is
$$
E(u)=\frac{1}{2} \int_{\R^{N}} (|(-\Delta)^{\frac{s}{2}}u|^{2}+u^{2}) \,dx-\int_{\R^{N}} k(x) F(u) \, dx- \int_{\R^{N}} h(x) u \, dx.
$$
\end{thm}
\begin{remark}
The assumption on the size of $h$ is a necessary condition to find a solution to $(\ref{P})$. In fact, proceeding as in \cite{CZ}, one can obtain a non-existence result to $(\ref{P})$ when $|h|_{2}$ is sufficiently large.
\end{remark}

\noindent
The proof of the above theorem goes as follows: under the assumption $l<\infty$, we first use the Ekeland variational principle to prove that for $|h|_{2}$ small enough, there exists a positive solution to (\ref{R0}) such that $I(U_{0})<0$. Then, we use a variant of Mountain Pass Theorem \cite{Ek}, to find a Cerami sequence which converges strongly in $\X$ to a solution $U_{1}$ of (\ref{R0}) with $I(U_{1})>0$. Clearly, these two solutions $U_{0}$ and $U_{1}$ are different.

\smallskip

\noindent

Our second result deals with the existence of positive solutions to (\ref{P}) in the superlinear case $l=\infty$. 
\begin{thm}\label{thm2}
Let $s\in (0,1)$ and $N> 2s$. Assume that  $f$ verifies $(f1)$-$(f4)$ with $l=\infty$. We also suppose that $k(x)\equiv 1$, $h\in C^{1}(\R^{N}) \cap L^{2}(\R^{N})$ is a radial function such that $h(x)\geq 0$, $h(x)\not \equiv 0$ and 
\begin{compactenum}[($H$)]
\item $x\cdot \nabla h(x)\in L^{1}(\R^{N})\cap L^{\infty}(\R^{N})$ and 
$$
x\cdot \nabla h(x)\geq 0 \mbox{ for all } x\in \R^{N}.
$$
\end{compactenum}
Let us assume that 
\begin{equation}
|h|_{2}<m_{1}:=\max_{t\geq 0}\left[\left(\frac{1}{2}-\frac{\varepsilon}{2}\right)t-\frac{C_{\varepsilon}}{p+1}S_{*}^{p+1}t^{p}\right],
\end{equation}
where $\varepsilon\in (0, 1)$ is fixed.
Then, (\ref{P}) admits two positive solutions $u_{3}, u_{4}\in H^{s}_{r}(\R^{N})$ such that $E(u_{3})<0<E(u_{4})$.
\end{thm}

Due to the presence of radial functions $k(x)=1$ and $h=h(|x|)$, we work in the subspace $\Y$ of the weight space $\X$, involving the functions which are radial with respect to $x\in \R^{N}$. 
We point out that the methods used to study the asymptotically linear case do not work any more. Indeed, to prove that a Palais-Smale sequence converges to a second solution different from the first one, we have to use the concentration-compactness principle which seems very hard to apply without requiring further assumptions on $k(x)$ and $f(t)$.\\
This time, we use the compactness of $\Y$ into $L^{q}(\R^{N})$ for any $q\in (2, 2^{*}_{s})$, and the Ekeland principle, to get a first solution to (\ref{R0}) with negative energy, provided that $|h|_{2}$ is sufficiently small. The existence of a second solution with positive energy is obtained by combining a generic result due to Jeanjean \cite{J}, which allows us to prove the existence of bounded Palais-Smale sequences for parametrized functionals, the Pohozaev identity for the fractional Laplacian and the assumption $(H)$, which guarantee the existence of a bounded Palais-Smale sequence for $I$, which converges to a radial positive solution to (\ref{R0}).
\smallskip

\noindent
The plan of the paper is the following:  in section $2$ we give some preliminaries which will be useful along the paper.
In section $3$ we consider the asymptotically linear case and we prove the existence of two positive solutions via mountain pass theorem. In section $4$ we study the superlinear case, and we provide the proof of Theorem \ref{thm2}. Finally, as applications of our results, we present some concrete examples.

\section{Preliminaries}

\noindent
In this section, we briefly recall some properties of the fractional Sobolev spaces, and we introduce some notations which we will used along the paper.\\
For any $s\in (0,1)$, we define $\mathcal{D}^{s, 2}(\R^{N})$ as the completion of $C^{\infty}_{0}(\R^{N})$ with respect to
$$
[u]^{2}=\iint_{\R^{2N}} \frac{|u(x)-u(y)|^{2}}{|x-y|^{N+2s}} \, dx \, dy =|(-\Delta)^{\frac{s}{2}} u|^{2}_{L^{2}(\R^{N})},
$$
that is
$$
\mathcal{D}^{s, 2}(\R^{N})=\left\{u\in L^{2^{*}_{s}}(\R^{N}): [u]<\infty\right\}.
$$
Now, let us introduce the fractional Sobolev space
$$
H^{s}(\R^{N})= \left\{u\in L^{2}(\R^{N}) : \frac{|u(x)-u(y)|}{|x-y|^{\frac{N+2s}{2}}} \in L^{2}(\R^{2N}) \right \}
$$
endowed with the natural norm 
$$
\|u\|_{H^{s}(\R^{N})} = \sqrt{[u]^{2} + |u|_{L^{2}(\R^{N})}^{2}}.
$$

\noindent
For the convenience of the reader we recall the following fundamental embeddings:
\begin{thm}\cite{DPV}\label{Sembedding}
Let $s\in (0,1)$ and $N>2s$. Then there exists a sharp constant $S_{*}=S(N, s)>0$
such that for any $u\in H^{s}(\R^{N})$
\begin{equation}\label{FSI}
|u|^{2}_{L^{2^{*}_{s}}(\R^{N})} \leq S_{*} [u]^{2}. 
\end{equation}
Moreover $H^{s}(\R^{N})$ is continuously embedded in $L^{q}(\R^{N})$ for any $q\in [2, 2^{*}_{s}]$ and compactly in $L^{q}_{loc}(\R^{N})$ for any $q\in [2, 2^{*}_{s})$. 
\end{thm}

\noindent
We also define the space of radial functions in $H^{s}(\R^{N})$
$$
H^{s}_{r}(\R^{N})=\left \{u\in H^{s}(\R^{N}): u(x)=u(|x|)\right\}. 
$$
Related to this space, the following compactness result due to Lions \cite{Lions} holds:
\begin{thm}\cite{Lions}\label{Lions}
Let $s\in (0,1)$ and $N\geq 2$. Then $H^{s}_{r}(\R^{N})$ is compactly in $L^{q}(\R^{N})$ for any $q\in (2, 2^{*}_{s})$.
\end{thm}

We also state the following useful result obtained in \cite{CW}:
\begin{lem}\cite{CW}\label{Strauss}
Let $(X, \|\cdot\|)$ be a Banach space such that $X$ is embedded respectively continuously and compactly into $L^{q}(\R^{N})$ for $q\in [q_{1}, q_{2}]$ and $q\in (q_{1}, q_{2})$, where $q_{1}, q_{2}\in (0, \infty)$.
Assume that $(u_{n})\subset X$, $u: \R^{N} \rightarrow \R$ is a measurable function and $P\in C(\R, \R)$ is such that
\begin{compactenum}[(i)]
\item $\displaystyle{\lim_{|t|\rightarrow 0} \frac{P(t)}{|t|^{q_{1}}}=0}$, \\
\item $\displaystyle{\lim_{|t|\rightarrow \infty} \frac{P(t)}{|t|^{q_{2}}}=0}$,\\
\item $\displaystyle{\sup_{n\in \N} \|u_{n}\|_{X}<\infty}$,\\
\item $\displaystyle{\lim_{n \rightarrow \infty} P(u_{n}(x))=u(x)} \mbox{ for a.e. } x\in \R^{N}$.
\end{compactenum}
Then, up to a subsequence, we have
$$
\lim_{n\rightarrow \infty} \|P(u_{n})-u\|_{L^{1}(\R^{N})}=0.
$$
\end{lem}

\noindent

Now, let us denote by $\mathcal{D}(\R^{N+1}_{+}, y^{1-2s})$ the completion of $C^{\infty}_{0}(\overline{\R^{N+1}_{+}})$ under the norm
$$
\|U\|^{2}_{\mathcal{D}(\R^{N+1}_{+}, y^{1-2s})}=\iint_{\R^{N+1}_{+}} y^{1-2s} |\nabla U|^{2} \, dx dy.
$$
It is known \cite{CS} that for any $U\in \mathcal{D}(\R^{N+1}_{+}, y^{1-2s})$, its trace $U(x, 0)$ belongs $\mathcal{D}^{s,2}(\R^{N})$ and 
that it is possible to define a trace continous map
\begin{equation}\label{Seok}
[U(\cdot, 0)]\leq C \|U\|_{\mathcal{D}(\R^{N+1}_{+}, y^{1-2s})}
\end{equation}
By combining (\ref{FSI}) and (\ref{Seok}), we can derive the following Sobolev inequality
\begin{equation}\label{Mia}
|U(\cdot, 0)|_{L^{2^{*}}_{s}(\R^{N})}\leq C \|U\|_{\mathcal{D}(\R^{N+1}_{+}, y^{1-2s})}
\end{equation}
for any $U\in \mathcal{D}(\R^{N+1}_{+}, y^{1-2s})$.\\
Then, as proved in \cite{CS}, for any $u\in \mathcal{D}^{s, 2}(\R^{N})$ there exists $U\in \mathcal{D}(\R^{N+1}_{+}, y^{1-2s})$, 
called the $s$-harmonic extension of $u$,  such that
\begin{equation}\label{Pext}
\left\{
\begin{array}{ll}
\dive(y^{1-2s} \nabla U)=0 & \mbox{ in } \R^{N+1}_{+} \\
U(x, 0)=u(x) & \mbox{ on } \R^{N}
\end{array}.
\right.
\end{equation}
Moreover,
\begin{equation*}
\frac{\partial U}{\partial \nu^{1-2s}}=-\lim_{y\rightarrow 0^{+}} y^{1-2s}\frac{\partial U}{\partial y}=(-\Delta)^{s}u(x)
\end{equation*}
and
\begin{equation*}
\int_{\R^{N+1}_{+}} y^{1-2s}|\nabla U|^{2} \,dx dy=\kappa_{s} [u]_{H^{s}(\R^{N})}^{2}.
\end{equation*}
Therefore, we can reformulate (\ref{P}) in a local way, and we can investigate the following extended problem in $\R^{N+1}_{+}$
\begin{equation}\label{R}
\left\{
\begin{array}{ll}
\dive(y^{1-2s} \nabla U)=0 & \mbox{ in } \R^{N+1}_{+} \\
\frac{\partial U}{\partial \nu^{1-2s}}=\kappa_{s} [-U(x, 0)+k(x) f(U(x, 0))+h(x)]  & \mbox{ on } \R^{N} \\
\end{array}.
\right.
\end{equation}
Qualitatively, the result of \cite{CS} states that one can localize the fractional Laplacian by adding an additional variable. This argument is fundamental to apply known variational methods. \\
At this point, we introduce the following functional space  
$$
\X=\{U\in \mathcal{D}(\R^{N+1}_{+}, y^{1-2s}): \int_{\R^{N}} |U(x, 0)|^{2}\, dx<\infty\}
$$
endowed with the norm 
$$
\|U\|^{2}_{\X}=\int_{\R^{N+1}_{+}} y^{1-2s}|\nabla U|^{2} \,dx dy+\int_{\R^{N}} |U(x, 0)|^{2}\, dx.
$$

\noindent
We recall that $\X$ is locally compactly embedded in the weight space $L^{2}(\R^{N+1}_{+}, y^{1-2s})$ endowed with the norm 
$$
\|U\|_{L^{2}(\R^{N+1}_{+}, y^{1-2s})}=\int_{\R^{N+1}_{+}} y^{1-2s}|U|^{2} \,dx dy.
$$
More precisely, we have 
\begin{lem}\cite{DMV}\label{DMV}
Let $R>0$ and let $\mathcal{T}$ be a subset of $\mathcal{D}(\R^{N+1}_{+}, y^{1-2s})$ such that 
$$
\sup_{U\in \mathcal{T}} \iint_{\R^{N+1}_{+}} y^{1-2s} |\nabla U|^{2} \, dx dy<\infty.
$$
Then $\mathcal{T}$ is pre-compact in $L^{2}(B_{R}^{+}, y^{1-2s})$, where $B_{R}^{+}=\{(x, y)\in \R^{N+1}_{+}: |(x, y)|<R\}$.\\
There exists a constant $C_{0}>0$ such that for all $U\in \mathcal{D}(\R^{N+1}_{+}, y^{1-2s})$ it holds
$$
\left(\iint_{\R^{N+1}_{+}} y^{1-2s} |U|^{2\gamma} \, dx dy\right)^{\frac{1}{2\gamma}}\leq C_{0} \left(\iint_{\R^{N+1}_{+}} y^{1-2s} |\nabla U|^{2} \, dx dy\right)^{\frac{1}{2}}
$$
where $\gamma=1+\frac{2}{N-2s}$.
\end{lem}

\begin{remark}
With abuse of notation, we will denote by $u$ the trace of a function $U\in \mathcal{D}(\R^{N+1}_{+}, y^{1-2s})$. Moreover, we denote by $|u|_{p}$ the $L^{p}$-norm of a function $u$ belonging to $L^{p}(\R^{N})$.\\
In what follows, for simplicity, we will omit the constant $\kappa_{s}$ appearing in the extended problem $(\ref{R})$.  
\end{remark}

\section{Asymptotically linear case: proof of Theorem \ref{thm1}}
\noindent
In this section we discuss the existence of positive solutions to (\ref{P}) under the assumption that $f$ is asymptotically linear.
Taking into account the results presented in Section $2$, we can consider the following degenerate elliptic problem
\begin{equation}\label{2.1}
\left\{
\begin{array}{ll}
\dive(y^{1-2s} \nabla U)=0 & \mbox{ in } \R^{N+1}_{+} \\
\frac{\partial U}{\partial \nu^{1-2s}}=-u+k(x) f(u)+h(x)  & \mbox{ on } \R^{N} \\
\end{array}
\right.
\end{equation}
where $k(x)$ is a bounded positive function, $h\in L^{2}(\R^{N})$, $h\geq 0$ ($h\not\equiv 0$) and $f$ satisfies $(f1)$-$(f4)$ with $l<+\infty$. 

\noindent

Since the proof of Theorem \ref{thm1}, consists of several steps, we first collect some useful lemmas.\\
We begin proving the following result:
\begin{lem}\label{lem2.1}
Suppose that $(f1)$-$(f4)$ with $l<+\infty$ hold. Let $h\in L^{2}(\R^{N})$, $k$ satisfies \eqref{1.3}, and $\{U_{n}\}\subset \X$ be a bounded (PS) sequence of $I$. Then $\{U_{n}\}$ has a strongly convergent subsequence in $\X$. 
\end{lem}

\begin{proof}
Firstly, we show that for any $\e>0$, there exist $R(\e)>R_{0}$ (where $R_{0}$ is given by (K)) and $n(\e)>0$ such that 
\begin{equation}\label{2.2}
\iint_{\R^{N+1}_{+} \setminus B_{R}^{+}} y^{1-2s} |\nabla U_{n}|^{2} \, dx dy+ \int_{\R^{N}\setminus B_{R}} u_{n}^{2} \, dx \leq \e, \quad \forall R\geq R(\e) \mbox{ and } n\geq n(\e). 
\end{equation}
Let $\Psi_{R}\in C^{\infty}(\R^{N+1}_{+})$ be a smooth function such that $0\leq \Psi_{R}\leq 1$, 
\begin{equation}\label{2.3}
\Psi_{R}(x, y)= 
\left\{
\begin{array}{ll}
0 \quad & (x, y)\in B^{+}_{\frac{R}{2}}\\
1 \quad &(x, y)\notin B^{+}_{R}. 
\end{array}
\right.
\end{equation}
and
\begin{equation}\label{2.4}
|\nabla \Psi_{R}(x, y)|\leq \frac{C}{R} \, \mbox{ for all } (x, y)\in \R^{N+1}_{+} 
\end{equation}
for some positive constant $C$ independent of $R$.

Then, we can observe that for any $U\in \X$ and all $R\geq 1$, there exists a constant $C_{1}>0$ such that 
$$
\|\Psi_{R} U\|_{\X}\leq C_{1}\|U\|_{\X}.
$$ 
Indeed, by using Young inequality and Lemma \ref{DMV}, we can see that 
\begin{align*}
&\iint_{\R^{N+1}_{+}} y^{1-2s} |\nabla (U \Psi_{R})|^{2} \, dx dy+\int_{\R^{N}} |u \psi_{R}|^{2} \, dx \nonumber \\
&\leq 2\iint_{\R^{N+1}_{+}} y^{1-2s} |\nabla U|^{2}\Psi_{R}^{2} \, dx dy+ 2\iint_{\R^{N+1}_{+}} y^{1-2s} |\nabla \Psi_{R}|^{2}U^{2} \, dx dy+\int_{\R^{N}} |u|^{2} \, dx \nonumber \\
&\leq 2 \iint_{\R^{N+1}_{+}} y^{1-2s} |\nabla U|^{2} \, dx dy +\frac{2C}{R^{2}} \iint_{B^{+}_{R}\setminus B^{+}_{\frac{R}{2}}} y^{1-2s} U^{2} \, dx dy +\int_{\R^{N}} |u|^{2} \, dx \nonumber \\
&\leq 2 \iint_{\R^{N+1}_{+}} y^{1-2s} |\nabla U|^{2} \, dx dy +\int_{\R^{N}} |u|^{2} \, dx+\frac{2C}{R^{2}} \Bigl(\iint_{B^{+}_{R}\setminus B^{+}_{\frac{R}{2}}} y^{1-2s}|\nabla U|^{2\gamma} \, dx dy \Bigr)^{\frac{1}{\gamma}} \Bigl(\iint_{B^{+}_{R}\setminus B^{+}_{\frac{R}{2}}} y^{1-2s} \, dx dy \Bigr)^{\frac{\gamma-1}{\gamma}} \nonumber \\
&\leq 2\left(1+C\right) \|U\|_{\X}^{2}\leq C_{1}\|U\|_{\X}^{2},
\end{align*}
where we used the facts 
$$
\iint_{B^{+}_{R}\setminus B^{+}_{\frac{R}{2}}} y^{1-2s} \, dx dy\leq C R^{N+2-2s} \,\mbox{ and } \, \frac{\gamma-1}{\gamma}=\frac{2}{N+2-2s}.
$$
Since $I'(U_{n})\rightarrow 0$ as $n\rightarrow \infty$ and $\{U_{n}\}$ is bounded in $\X$, we know that, for any $\e>0$, there exists $n(\e)>0$ such that
\begin{equation*}
\langle I'(U_{n}), \Psi_{R} U_{n} \rangle \leq C_{1} \|I'(U_{n})\| \|U_{n}\|_{\X}\leq \frac{\e}{4}, \mbox{ for } n\geq n(\e). 
\end{equation*}
Equivalently, for all $n\geq n(\e)$, we get
\begin{align}\label{2.5}
&\iint_{\R^{N+1}_{+}} y^{1-2s}|\nabla U_{n}|^{2}\Psi_{R} \, dx dy+ \int_{\R^{N}} u_{n}^{2}\psi_{R} \, dx \nonumber \\
&\leq \int_{\R^{N}} (k(x) f(u_{n}) + h(x)) u_{n} \psi_{R} \, dx-\iint_{\R^{N+1}_{+}} y^{1-2s}\nabla U_{n} \nabla \psi_{R} U_{n} \, dx dy + \frac{\e}{4}. 
\end{align}
Now, by using $(f1)$ and \eqref{1.3}, we obtain that there exists $0<\theta <1$ such that
\begin{equation}\label{Khic}
k(x) f(u_{n}) u_{n} \leq \theta u_{n}^{2} \quad \mbox{ for } |x|\geq R_{0}. 
\end{equation}
Since $h\in L^{2}(\R^{N})$ and $\|U_{n}\|_{\X}\leq C$ for some constant $C>0$, it follows from \eqref{2.3} there exists $R(\e)>R_{0}$ such that 
\begin{equation}\label{Hhic}
\int_{\R^{N}} h(x) u_{n} \psi_{R}\, dx \leq |h(x) \psi_{R}|_{2} |u_{n}|_{2} \leq \frac{\e}{4}, \quad \mbox{ for } R\geq R(\e). 
\end{equation}
Due to the boundedness of $\{U_{n}\}$ in $\X$, we may assume, up to a subsequence, that there exists $U\in \X$ such that $U_{n}\rightharpoonup U$ in $\X$, $u_{n}\rightarrow u$ in $L^{q}_{loc}(\R^{N})$ for any $q\in [2, 2^{*}_{s})$ and $u_{n}\rightarrow u$ a.e. in $\R^{N}$.\\
Therefore, \eqref{2.4}, $\|U_{n}\|_{\X}\leq C$, H\"older inequality and Lemma \ref{DMV} yield
\begin{align}\label{2.5WZ}
&\lim_{R\rightarrow \infty}\limsup_{n\rightarrow \infty}\left|\iint_{\R^{N+1}_{+}} y^{1-2s}\nabla U_{n} \nabla \Psi_{R} U_{n} \, dx dy\right| \nonumber \\
& \leq \lim_{R\rightarrow \infty}\limsup_{n\rightarrow \infty}\frac{C}{R} \left(\iint_{B^{+}_{R}\setminus B^{+}_{\frac{R}{2}}} y^{1-2s}|\nabla U_{n}|^{2} \, dx dy \right)^{\frac{1}{2}} 
\left(\iint_{B^{+}_{R}\setminus B^{+}_{\frac{R}{2}}} y^{1-2s}|U_{n}|^{2} \, dx dy \right)^{\frac{1}{2}} \nonumber \\
&\leq \lim_{R\rightarrow \infty} \frac{C}{R}
\left(\iint_{B^{+}_{R}\setminus B^{+}_{\frac{R}{2}}} y^{1-2s}|U|^{2} \, dx dy \right)^{\frac{1}{2}} \nonumber \\
&\leq  \lim_{R\rightarrow \infty} \frac{C}{R} 
\left(\iint_{B^{+}_{R}\setminus B^{+}_{\frac{R}{2}}} y^{1-2s}|U|^{2\gamma} \, dx dy \right)^{\frac{1}{2\gamma}} \left(\iint_{B^{+}_{R}\setminus B^{+}_{\frac{R}{2}}} y^{1-2s} \, dx dy \right)^{\frac{\gamma-1}{2\gamma}}\nonumber \\
&\leq C \lim_{R\rightarrow \infty} \left(\iint_{B^{+}_{R}\setminus B^{+}_{\frac{R}{2}}} y^{1-2s}|U|^{2\gamma} \, dx dy \right)^{\frac{1}{2\gamma}} =0.
\end{align}
Then, putting together \eqref{2.5}, \eqref{Khic}, \eqref{Hhic} and \eqref{2.5WZ},  we have for any $R\geq R(\e)$ and $n\geq n(\e)$ sufficiently large
\begin{equation}\label{Eval}
\iint_{\R^{N+1}_{+}} y^{1-2s}|\nabla U_{n}|^{2} \Psi_{R}\, dx dy + \int_{\R^{N}} (1-\theta)u_{n}^{2}\psi_{R} \, dx \leq \e. 
\end{equation}
From $\theta\in (0, 1)$ and \eqref{2.3}, we can deduce that \eqref{Eval} implies \eqref{2.2}.\\
Now, we exploit the relation (\ref{2.2}) in order to prove the existence of a convergent subsequence for $\{U_{n}\}$.
By using the fact that $I'(U_{n})=0$ and $\{U_{n}\}$ is bounded in $\X$, we can see that 
\begin{equation}\label{AH1}
\langle I'(U_{n}), U_{n}\rangle=\iint_{\R^{N+1}_{+}} y^{1-2s}|\nabla U_{n}|^{2} \, dx dy+\int_{\R^{N}}u^{2}_{n} \, dx-\int_{\R^{N}} k(x) f(u_{n}) u_{n}\, dx-\int_{\R^{N}} h(x)u_{n}\, dx=o(1)
\end{equation}
and
\begin{equation}\label{AH2}
\langle I'(U_{n}), U\rangle=\iint_{\R^{N+1}_{+}} y^{1-2s}\nabla U_{n} \nabla U\, dx dy+\int_{\R^{N}}u_{n} u\, dx-\int_{\R^{N}} k(x) f(u_{n}) u\, dx-\int_{\R^{N}} h(x)u\, dx=o(1).
\end{equation}
Hence, in order to prove our Lemma, it is enough to prove that $\|U_{n}\|_{\X}\rightarrow \|U\|_{\X}$ as $n\rightarrow \infty$. In view of (\ref{AH1}) and (\ref{AH2}), this is equivalent to show that 
\begin{equation}\label{AH3}
\int_{\R^{N}} k(x) f(u_{n})(u_{n}-u)\, dx+\int_{\R^{N}} h(x) (u_{n}-u)\, dx=o(1).
\end{equation}
Clearly, by using the facts $k\in L^{\infty}(\R^{N})$, $h\in L^{2}(\R^{N})$ and $u_{n}\rightarrow u$ in $L^{2}(B_{R})$ for any $R>0$, we can see that 
\begin{equation}\label{AH4}
\int_{B_{R}} k(x) f(u_{n})(u_{n}-u)\, dx+\int_{B_{R}} h(x) (u_{n}-u)\, dx=o(1).
\end{equation}
On the other hand, by using (\ref{2.2}), we know that for any $\e>0$ there exists $R(\e)>0$ such that
\begin{align}\label{AH5} 
&\int_{|x|\geq R(\e)} k(x) f(u_{n})(u_{n}-u)\, dx+\int_{\R^{N}} h(x) (u_{n}-u)\, dx \nonumber \\
&\leq \left(\int_{|x|\geq R(\e)} k(x) |f(u_{n})|^{2}\, dx \right)^{\frac{1}{2}}\left(\int_{|x|\geq R(\e)} k(x) |u_{n}-u|^{2}\, dx \right)^{\frac{1}{2}} \nonumber \\
&+\left(\int_{|x|\geq R(\e)} |h(x)|^{2} \, dx \right)^{\frac{1}{2}}\left(\int_{|x|\geq R(\e)} |u_{n}-u|^{2}\, dx \right)^{\frac{1}{2}} \nonumber \\
&\leq C \left(\int_{|x|\geq R(\e)} |u_{n}|^{2}\, dx \right)^{\frac{1}{2}} \left(\int_{|x|\geq R(\e)} |u_{n}-u|^{2}\, dx \right)^{\frac{1}{2}}+|h|_{2}\left(\int_{|x|\geq R(\e)} |u_{n}-u|^{2}\, dx \right)^{\frac{1}{2}} \nonumber \\
&\leq C\e
\end{align}
for $n$ large enough.
By combining (\ref{AH4}) and (\ref{AH5}) we obtain (\ref{AH3}). This concludes the proof of lemma.

\end{proof}

\noindent
In the next Lemma we show that  $I$ is positive on the boundary of some ball in $\X$, provided that $|h|_{2}$ is sufficiently small.
This property will be fundamental to apply Ekeland's variational principle. 

\begin{lem}\label{lem2.2}
Let us assume that $(f1)$-$(f3)$ hold, $h\in L^{2}(\R^{N})$ such that (\ref{HV}) is satisfied, and $k\in L^{\infty}(\R^{N})$. Then there exist $\rho, \alpha, m>0$ such that $I(U)|_{\|U\|= \rho}\geq \alpha >0$ for $|h|_{2}<m$. 
\end{lem}

\begin{proof}
Fix $\varepsilon\in (0, |k|^{-1}_{L^{\infty}(\R^{N})})$. Then, in view of  Theorem \ref{FSE} and (\ref{2.7}), we get
\begin{align}\begin{split}\label{2.8}
I(U)&\geq \frac{1}{2} \|U\|^{2} - \frac{\e}{2}|k|_{L^{\infty}(\R^{N})} \|U\|^{2} - \frac{C(\e)}{p+1}|k|_{L^{\infty}(\R^{N})} S_{*}^{p+1} \|U\|^{p+1} - |h|_{2} \|U\|\\
&= \|U\| \left[ \left( \frac{1}{2} - C_{1}\e\right) \|U\| - C_{2}(\e) \|U\|^{p} - |h|_{2} \right],  
\end{split}\end{align}
where
$$
C_{1}:=\frac{1}{2}|k|_{L^{\infty}(\R^{N})}  \mbox{ and } C_{2}(\varepsilon):=\frac{C(\e)}{p+1}|k|_{L^{\infty}(\R^{N})} S_{*}^{p+1}.
$$
By using (\ref{HV}) and (\ref{2.8}), we can infer that there exist $\rho, \alpha>0$ such that $I(U)|_{\|U\|=\rho}\geq \alpha$ provided that $|h|_{2}<m$.

\end{proof}

For $\rho$ given by Lemma \ref{lem2.2}, we denote $B_{\rho}= \{U \in \X: \|U\|<\rho \}$ the ball in $\X$ with center in $0$ and radius $\rho$. 
As a consequence of Ekeland's variational principle and Lemma \ref{lem2.1}, we can see that $I$ has a local minimum if $|h|_{2}$ is  small enough.

\begin{thm}\label{thm2.1}
Assume that $(f1)$-$(f4)$ with $l<+\infty$ hold, $h\in L^{2}(\R^{N})$, $h\geq 0$ ($h\not\equiv 0$) and $k$ satisfies \eqref{1.3}. If $|h|_{2}<m$, $m$ is given by Lemma $\ref{lem2.2}$, then there exists $U_{0}\in \X$ such that
\begin{equation*}
I(U_{0})= \inf \{I(U) : U\in \overline{B}_{\rho} \}<0,
\end{equation*}
and $U_{0}$ is a positive solution of problem \eqref{2.1}. 
\end{thm}

\begin{proof}
Since $h(x)\in L^{2}(\R^{N})$, $h\geq 0$ and $h\not\equiv 0$, we can choose a function $V \in \X$ such that
\begin{equation}\label{2.9}
\int_{\R^{N}} h(x) v(x)\, dx>0. 
\end{equation}
For all $t>0$, we can note that 
\begin{align*}
I(tV)&= \frac{t^{2}}{2} \left[\int_{\R^{N+1}_{+}} y^{1-2s} |\nabla V|^{2}\, dx dy +\int_{\R^{N}} v^{2} \, dx \right] - \int_{\R^{N}} k(x) F(t v) \, dx - t\int_{\R^{N}} h(x) v(x)\, dx \\
&\leq  \frac{t^{2}}{2} \|V\|^{2} - t\int_{\R^{N}} h(x) v(x)\, dx<0 \, \mbox{ for } t>0 \mbox{ small enough. }
\end{align*}
Then
$$
c_{0}:= \inf\{ I(U): U\in \overline{B}_{\rho}\}<0.
$$ 
By applying the Ekeland's variational principle, we know that there exists $\{U_{n}\}\subset \overline{B}_{\rho}$ such that
\begin{compactenum}[$(i)$]
\item $c_{0}\leq I(U_{n})<c_{0} +\frac{1}{n}$, 
\item $I(W)\geq I(U_{n}) - \frac{1}{n} \|W-U_{n}\|$ for all $W\in \overline{B}_{\rho}$. 
\end{compactenum}
Now, our claim is to prove that $\{U_{n}\}$ is a bounded (PS) sequence of $I$. 

Firstly, we show that $\|U_{n}\|<\rho$ for a $n$ large enough. If $\|U_{n}\|=\rho$ for infinitely many $n$, then we may assume that $\|U_{n}\|=\rho$ for all $n\geq 1$. Hence, by Lemma \ref{lem2.2}, we can see that $I(U_{n})\geq \alpha>0$. Taking the limit as $n\rightarrow \infty$ and by using $(i)$, we can deduce that $0>c_{0}\geq \alpha>0$, which is a contradiction. 

Now, we show that $I'(U_{n})\rightarrow 0$. Indeed, for any $U\in \X$ with $\|U\|=1$, let $W_{n}=U_{n}+t U$. For a fixed $n$, we have $\|W_{n}\|\leq \|U_{n}\|+t<\rho$ when $t$ is small enough. By using $(ii)$, we deduce that 
$$
I(W_{n})\geq I(U_{n})-\frac{t}{n}\|U\|,
$$ 
that is 
$$
\frac{I(W_{n})-I(U_{n})}{t}\geq -\frac{\|U\|}{n}=-\frac{1}{n}.
$$ 
Taking the limit as $t\rightarrow 0$, we deduce that $\langle I'(U_{n}), U\rangle\geq -\frac{1}{n}$, which means $|\langle I'(U_{n}), U\rangle|\leq \frac{1}{n}$ for any $U\in \X$ with $\|U\|=1$.
This shows that $\{U_{n}\}$ is a bounded (PS) sequence of $I$.
Then, by using Lemma \ref{lem2.1}, we can see that there exists $U_{0}\in \X$ such that $I'(U_{0})=0$ and $I(U_{0})=c_{0}<0$. 

\end{proof}

\noindent
In what follows, we show that problem \eqref{2.1} has a mountain pass type solution. In order to do this, we use the following variant of version of Mountain Pass Theorem which allows us to find a so-called Cerami sequence $\{U_{n}\}$. 
Since this type of Palais-Smale sequence enjoys of some useful properties, we are able to prove its boundedness in the asymptotically linear case.
\begin{thm}\cite{Ek}\label{Ek}
Let $X$ be a real Banach space with its dual $X^{*}$, and suppose that $I\in C^{1}(X, \R)$ satisfies
$$
\max\{I(0), I(e)\}\leq \mu<\alpha\leq \inf_{\|x\|=\rho} I(x),
$$
for some $\mu<\alpha$, $\rho>0$ and $e\in X$ with $\|e\|>\rho$. Let $c\geq \alpha$ be characterized by
$$
c=\inf_{\gamma\in \Gamma}\max_{t\in [0, 1]} I(\gamma(t)),
$$
where
$$
\Gamma=\{\gamma\in C([0, 1], X): \gamma(0)=0, \gamma(1)=e\}.
$$
Then, there exists a Cerami sequence $\{x_{n}\}\subset X$ at the level $c$ that is
$$
I(x_{n})\rightarrow c \mbox{ and } (1+\|x_{n}\|)\|I'(x_{n})\|_{*}\rightarrow 0
$$ 
as $n\rightarrow \infty$.
\end{thm}

\noindent
The below lemma, shows that $I$ possesses a mountain pass geometry. 
\begin{lem}\label{lem2.3}
Suppose that $(f1)$- $(f4)$ hold and $\mu^{*}\in (l, +\infty)$ with $\mu^{*}$ given by \eqref{1.4}. Then there exists $V\in \X$ with $\|V\|>\rho$, $\rho$ is given by Lemma $\ref{lem2.2}$, such that $I(V)<0$.  
\end{lem}

\begin{proof}
Being $l>\mu^{*}$, we can find a nonnegative function $W\in \X$, such that
\begin{equation*}
\int_{\R^{N}} k(x)w^{2} \, dx =1 \mbox{ such that } \int_{\R^{N+1}_{+}} y^{1-2s}|\nabla W|^{2} \, dx dy + \int_{\R^{N}} w^{2} \, dx <l. 
\end{equation*}
By using $(f4)$ and Fatou's lemma, we can see that
\begin{equation*}
\lim_{t\rightarrow \infty}\frac{I(tW)}{t^{2}}= \frac{1}{2}\|W\|^{2} - \lim_{t\rightarrow \infty} \int_{\R^{N}} k(x) \frac{F(tw)}{t^{2}}\, dx - \lim_{t\rightarrow \infty} \frac{1}{t}\int_{\R^{N}} h(x)w(x) \, dx \leq \frac{1}{2} (\|W\|^{2} -l)<0
\end{equation*}
Then, we take $V=t_{0}W$ with $t_{0}$ large enough. 

\end{proof}

\noindent
Putting together Lemmas \ref{lem2.2} and Lemma \ref{lem2.3}, we can see that the assumptions of Theorem \ref{Ek} are satisfied. Then, we can find a sequence $\{U_{n}\}\subset \X$ with the following property
\begin{equation}\label{2.10}
I(U_{n})\rightarrow c>0 \mbox{ and } \|I'(U_{n})\|(1+ \|U_{n}\|)\rightarrow 0,
\end{equation}
Let
\begin{equation*}
W_{n}=\frac{U_{n}}{\|U_{n}\|}. 
\end{equation*}
Obviously, $\{W_{n}\}$ is bounded in $\X$, so there exists a $W\in \X$ such that, up to a subsequence, we have
\begin{align}\begin{split}\label{2.11}
&W_{n}\rightharpoonup W \mbox{ in } \X, \\
&w_{n}\rightarrow w \mbox{ a.e. in } \R^{N}, \\
&w_{n}\rightarrow w \mbox{ strongly in } L^{2}_{loc}(\R^{N}).  
\end{split}\end{align}

\noindent
With the notation above introduced, we prove the following lemma. 
\begin{lem}\label{lem2.4}
Assume that $(f1)$-$(f4)$ and $(K)$ hold. Let $h\in L^{2}(\R^{N})$ and $\mu^{*}\in (l, +\infty)$ for $\mu^{*}$ given by \eqref{1.4}. If $\|U_{n}\|\rightarrow \infty$, then $W$ given by \eqref{2.11} is a nontrivial nonnegative solution of 
\begin{equation}\label{2.12}
\left\{
\begin{array}{ll}
\dive(y^{1-2s} \nabla W)=0 & \mbox{ in } \R^{N+1}_{+} \\
\frac{\partial W}{\partial \nu^{1-2s}}=-w+ l k(x) w  & \mbox{ on } \R^{N} \\
\end{array}
\right.  
\end{equation} 
\end{lem}

\begin{proof}
Firstly, we show that $W \not \equiv 0$. We argue by contradiction, and we assume that $W\equiv 0$. \\
By using the Sobolev embedding, we can see that $w_{n}\rightarrow 0$ strongly in $L^{2}(B_{R_{0}})$ where $R_{0}$ is given by (K). 
On the other hand, by $(f1)$, $(f4)$ and $l<+\infty$, we can find  $C>0$ such that
\begin{equation}\label{2.13}
\frac{f(t)}{t}\leq C, \mbox{ for all } t\in \R. 
\end{equation} 
Therefore, we can deduce that
\begin{equation}\label{2.14}
\int_{|x|<R_{0}} k(x) \frac{f(u_{n})}{u_{n}} w_{n}^{2} \, dx \leq C|k|_{\infty} \int_{|x|<R_{0}} w_{n}^{2} \, dx \rightarrow 0. 
\end{equation}
By the condition $(K)$, we can find $\eta \in (0,1)$ such that 
\begin{equation}\label{2.15}
\sup \left\{\frac{f(t)}{t} : t>0\right\} < \eta \inf \left\{ \frac{1}{k(x)} : |x|\geq R_{0}\right\}, 
\end{equation}
so, for all $n\in \N$, we get
\begin{equation}\label{2.16}
\int_{|x|\geq R_{0}} k(x) \frac{f(u_{n})}{u_{n}} |w_{n}|^{2}\, dx \leq \eta \int_{|x|\geq R_{0}} |w_{n}|^{2} \, dx \leq \eta <1. 
\end{equation}
Putting together \eqref{2.14} and \eqref{2.16}, we have 
\begin{equation}\label{2.17}
\limsup_{n\rightarrow \infty} \int_{\R^{N}} k(x)\frac{f(u_{n})}{u_{n}} w_{n}^{2} \, dx <1. 
\end{equation}
Now, by using the fact that $\|U_{n}\|\rightarrow \infty$ and \eqref{2.10}, we can see that
\begin{equation*}
\frac{\langle I'(U_{n}) , U_{n}\rangle}{\|U_{n}\|^{2}}= o(1), 
\end{equation*}
that is 
\begin{equation*}
o(1)= \|W_{n}\|^{2}- \int_{\R^{N}} k(x) \frac{f(u_{n})}{u_{n}} w_{n}^{2} \, dx = 1- \int_{\R^{N}} k(x) \frac{f(u_{n})}{u_{n}} w_{n}^{2} \, dx, 
\end{equation*}
which yields a contradiction in view of \eqref{2.17}. Then, we have proved that $W\not \equiv 0$. 

In what follows, we will show that $W$ is nonnegative, that is $W\geq 0$. 
Let $W_{n}^{-}(x)=\max \{ - W_{n}(x), 0\}$, and we observe that $\{W_{n}^{-}\}$ is bounded in $\X$. 

Since $\|U_{n}\|\rightarrow \infty$, we obtain that
\begin{equation*}
\frac{\langle I'(U_{n}), W_{n}^{-} \rangle}{\|U_{n}\|}=o(1), 
\end{equation*}
which gives 
\begin{equation}\label{2.18}
-\|W_{n}^{-}\|^{2} = \int_{\R^{N}} k(x)\frac{f(u_{n})}{\|u_{n}\|} w_{n}^{-} \, dx + o(1). 
\end{equation}
Taking into account $(f1)$, we know that $f(t)\equiv 0$ for all $t\leq 0$, so \eqref{2.18} implies that 
$$
\lim_{n\rightarrow \infty}\|W_{n}^{-}\|=0,
$$ 
which gives $W^{-}=0$ a.e. $x\in \R^{N}$, that is $W\geq 0$. 

Finally, we prove that $W$ is a solution to \eqref{2.12}. By using \eqref{2.10} and $\|U_{n}\|\rightarrow \infty$, we get
\begin{equation*}
\frac{\langle I'(U_{n}), \varPhi \rangle}{\|U_{n}\|}=o(1), \mbox{ for any } \varPhi \in C^{\infty}_{0}(\R^{N}), 
\end{equation*} 
or explicitly
\begin{equation}\label{2.19}
\iint_{\R^{N+1}_{+}} y^{1-2s}\nabla W_{n} \nabla \varPhi \, dx dy+\int_{\R^{N}} w_{n} \phi \, dx = \int_{\R^{N}} k(x)\frac{f(u_{n})}{u_{n}} w_{n} \phi \, dx + o(1)
\end{equation}
where we have used the notation $\phi=\varPhi(\cdot, 0)$.
Since $W_{n}\rightharpoonup W$ in $\X$ and $w_{n}\rightarrow w$ in $L^{2}_{loc}(\R^{N})$, we can deduce that
\begin{equation}\label{2.20}
\iint_{\R^{N+1}_{+}} y^{1-2s}\nabla W_{n} \nabla \varPhi \, dx dy+\int_{\R^{N}} w_{n} \phi \, dx = \int_{\R^{N}} k(x)\frac{f(u_{n})}{u_{n}} w_{n}\phi \, dx + o(1). 
\end{equation}
As a consequence, to prove that $W$ solves \eqref{2.12}, it is suffices to show that
\begin{equation}\label{2.21}
\int_{\R^{N}} k(x)\frac{f(u_{n})}{u_{n}} w_{n}(x)\phi(x)\, dx \rightarrow \int_{\R^{N}} l k(x) w(x) \phi(x)\, dx.  
\end{equation}
Firstly, we note that by \eqref{2.13} and $\|W_{n}\|=1$ we get
\begin{align*}
\int_{\R^{N}} \left|\frac{f(u_{n})}{u_{n}} w_{n}(x)\right|^{2}\, dx\leq C \int_{\R^{N}} w^{2}_{n}\, dx\leq C\|W_{n}\|^{2}=C
\end{align*}
that is $\{\frac{f(u_{n})}{u_{n}} w_{n}\}$ is bounded in $L^{2}(\R^{N})$.

Now, let us define the following sets
\begin{equation*}
\Omega_{+}=\{x\in \R^{N} : w(x)>0\} \, \mbox{ and } \, \Omega_{0}=\{x\in \R^{N} : w(x)=0\}. 
\end{equation*}
In view of \eqref{2.11}, it is clear that $u_{n}(x)\rightarrow +\infty$ a.e. in $x\in \Omega_{+}$ . 
Then, by $(f4)$, it follows that
\begin{equation}\label{AH6}
\frac{f(u_{n})}{u_{n}} w_{n}(x) \rightarrow l w(x) \, \mbox{ a.e. in } x\in \Omega_{+}. 
\end{equation}
Since $w_{n}\rightarrow 0$ a.e. in $x\in \Omega_{0}$, from \eqref{2.13} we obtain that
\begin{equation}\label{AH7}
\frac{f(u_{n})}{u_{n}} w_{n}(x) \rightarrow 0\equiv l w(x) \, \mbox{ a.e. in } x\in \Omega_{0}.
\end{equation}
Putting together (\ref{AH6}) and (\ref{AH7}), we can deduce that
\begin{equation}\label{2.22}
\frac{f(u_{n})}{u_{n}} w_{n}(x) \rightharpoonup l w(x) \, \mbox{ in }  L^{2}(\R^{N}). 
\end{equation}
Now, by using the facts $\phi \in C^{\infty}_{0}(\R^{N})$ and $k\in L^{\infty}(\R^{N})$, we can see that $z(x)= k(x)\phi(x) \in L^{2}(\R^{N})$, and this together with \eqref{2.22} implies that
\begin{equation*}
\int_{\R^{N}} \frac{f(u_{n})}{u_{n}} w_{n}(x) z(x) \rightarrow \int_{\R^{N}} l w(x) z(x)\, dx \mbox{ as } n\rightarrow \infty,
\end{equation*}
that is \eqref{2.21} holds.  

\end{proof}

\begin{lem}\label{lem2.5}
If $k\in L^{\infty}(\R^{N}, \R^{+})$ and let $\mu^{*}$ be defined by (1.4) with $l\in (\mu^{*}, +\infty)$. Then, \eqref{2.12} has no any nontrivial nonnegative solution. 
\end{lem}

\begin{proof}
Since $l>\mu^{*}$, there is a constant $\delta>0$ such that $\mu^{*}<\mu^{*} +\delta <l$. By the definition of $\mu^{*}$, there exists $V_{\delta}\in \X$ such that $\int_{\R^{N}} K(x) v_{\delta}^{2}(x) \, dx =1$ and 
\begin{equation*}
\mu^{*}\leq \|V_{\delta}\|^{2}<\mu^{*}+\delta. 
\end{equation*}
Since $C^{\infty}_{0}(\R^{N+1}_{+})$ is dense in $\X$, we may assume $V_{\delta}\in C^{\infty}_{0}(\R^{N+1}_{+})$. Let $R>0$ be such that $\supp V_{\delta}\subset B^{+}_{R}$ and define 
\begin{equation*}
\mu_{R}= \inf \left\{\iint_{B^{+}_{R}} y^{1-2s} |\nabla U|^{2}\, dx dy +\int_{\Gamma^{0}_{R}} u^{2} \, dx : \int_{\Gamma^{0}_{R}} K(x)u^{2}(x)\, dx=1, U\in H^{1}_{\Gamma^{+}_{R}}(B^{+}_{R}) \right\}, 
\end{equation*}
where we used the following notations
$$
B_{R}^{+}=\{(x, y)\in \R^{N+1}: y>0, |(x, y)|< R\},
$$
$$
\Gamma_{R}^{+}=\{(x, y)\in \R^{N+1}: y\geq 0, |(x, y)|= R\},
$$
$$
\Gamma_{R}^{0}=\{(x, 0)\in \partial \R^{N+1}_{+}: |x|< R\},
$$
and
$$
H^{1}_{\Gamma^{+}_{R}}(B^{+}_{R})=\{V\in H^{1}(B_{R}^{+}, y^{1-2s}): V\equiv 0 \mbox{ on } \Gamma_{R}^{+}\}.
$$
Since $V_{\delta}\equiv 0$ on $\Gamma_{R}^{+}$,
we can infer that $V_{\delta}\in H^{1}_{\Gamma_{R}^{+}}(B^{+}_{R})$ and 
\begin{equation}\label{new2.20}
\mu_{R}\leq \|V_{\delta}\|^{2} <\mu^{*}+\delta<l. 
\end{equation}
By the compactness of the embedding $H^{1}_{\Gamma_{R}^{+}}(B^{+}_{R})\subset L^{2}(\Gamma^{0}_{R})$, it is not difficult to see that there exists $W_{R}\in H^{1}_{\Gamma_{R}^{+}}(B^{+}_{R})\setminus \{0\}$ with $W_{R}\geq 0$ and $\int_{\Gamma^{0}_{R}} K(x)w_{R}^{2} (x)\, dx=1$ such that
\begin{equation}\label{new2.21}
\left\{
\begin{array}{ll}
\dive(y^{1-2s} \nabla W_{R})=0 & \mbox{ in } B^{+}_{R} \\
\frac{\partial W_{R}}{\partial \nu^{1-2s}}=-w_{R}+ \mu_{R} k(x) w_{R}  & \mbox{ on } \Gamma_{R}^{0}  \\
W^{R}=0 & \mbox{ on } \Gamma_{R}^{+}.
\end{array}
\right. 
\end{equation}
It follows from the strong maximum principle \cite{CabS} that $W_{R}>0$ on $B_{R}^{+}$. We extend $W_{R}=0$ in $\R^{N+1}_{+}\setminus B_{R}^{+}$, so that $W_{R}\in \X$.
Therefore, if $U\not \equiv 0$, $U\in \X$ is a nonnegative solution of \eqref{2.12}, then
\begin{align}\begin{split}\label{new2.22}
\mu_{R} \int_{\Gamma^{0}_{R}} k(x) w_{R} u \, dx &= \iint_{B^{+}_{R}} y^{1-2s} \nabla W_{R}\nabla U \, dx dy +\int_{\Gamma_{R}^{0}} w_{R} u\, dx\\
&= l \int_{\Gamma^{0}_{R}} k(x) u w_{R} \, dx.   
\end{split}\end{align}
Using $u\geq 0$ and $u\not\equiv 0$, we may choose $R>0$ large enough such that $\int_{\Gamma_{R}^{0}} K(x) u w_{R}\, dx>0$. \\
Then, \eqref{new2.22} implies that $\mu_{R}= l$, which is a contradiction in view of \eqref{new2.20}. 
\end{proof}

\begin{proof}[Proof of Theorem 1.1]
By using Lemmas \ref{lem2.4} and \ref{lem2.5}, it is obvious that the situation $\|U_{n}\|\rightarrow \infty$ cannot occur. Therefore, the sequence $\{U_{n}\}$ is bounded in $\X$. Taking into account Lemma \ref{lem2.1} and the Harnack inequality \cite{CabS}, we can deduce that problem \eqref{2.1} admits a positive solution $U_{1}\in \X$ with $I(U_{1})>0$. Then, the thesis of theorem follows by  Theorem \ref{thm2.1}. 

\end{proof}

\section{Superlinear case: proof of Theorem \ref{thm2}}
\noindent

This section is devoted to the proof of Theorem \ref{thm2}.
For simplicity, we consider problem (\ref{R}) with $k(x)\equiv 1$, that is 
\begin{equation}\label{3.1}
\left\{
\begin{array}{ll}
\dive(y^{1-2s} \nabla U)=0 & \mbox{ in } \R^{N+1}_{+} \\
\frac{\partial U}{\partial \nu^{1-2s}}=-u+ f(u)+h(x)  & \mbox{ on } \R^{N} \\
\end{array}
\right.
\end{equation}
where $h(x)= h(|x|)\in C^{1}(\R^{N})\cap L^{2}(\R^{N})$, $h(x)\geq 0$, $h(x)\not \equiv 0$ and $f$ satisfies $(f1)$-$(f4)$ with $l=+\infty$. Since we assume that $k(x)\equiv 1$ and $h(x)$ is radial, it is natural to work on the space of the function belonging to $\X$ which are radial with respect to $x$, that is
$$
\Y=\left\{U\in \X: U(x, y)=U(|x|, y)\right\}.
$$ 
We begin proving the following preliminary result
\begin{thm}\label{thm3.1}
Suppose that $h(x)= h(|x|)\in L^{2}(\R^{N})$, $h(x)\geq 0$, $h(x)\not \equiv 0$ and conditions $(f1)$-$(f3)$ holds, then there exist $m_{1}>0$ and $\tilde{U}_{0}\in \Y$ such that $I'(\tilde{U}_{0})=0$ and $I(\tilde{U}_{0})<0$ if $|h|_{2}<m_{1}$. 
\end{thm}

\begin{proof}
Arguing as in the proof of Theorem \ref{thm2.1}, by applying the Ekeland's variational principle, we can obtain the existence of a bounded (PS) sequence $\{\tilde{U}_{n}\}\subset \Y$ such that
\begin{equation*}
I(\tilde{U}_{n})\rightarrow \tilde{c}_{0}:= \inf \{I(U): U\in \Y \mbox{ and } \|U\|=\rho \}<0, 
\end{equation*}
where $\rho$ is given by Lemma \ref{lem2.2}. We claim that such infimum is achieved. \\
By using Theorem \ref{Lions}, we may assume that there exists $\tilde{U}_{0}\in \Y$ such that $\tilde{U}_{n}\rightharpoonup \tilde{U}_{0}$ in $\Y$, $\tilde{u}_{n}\rightarrow \tilde{u}_{0}$ in $L^{p+1}(\R^{N})$. 

Taking into account $(f1)$-$(f3)$, Theorem \ref{FSI}, and by exploiting the fact that $\{U_{n}\}$ is bounded in $\Y$, we can see that 
\begin{align*} 
\left|\int_{\R^{N}} f(\tilde{u}_{n})(\tilde{u}_{n}-\tilde{u}_{0}) \, dx\right|\leq \e|\tilde{u}_{n}|_{2} |\tilde{u}_{n}-\tilde{u}_{0}|_{2}+C_{\e}|\tilde{u}_{n}|^{p}_{p+1}|\tilde{u}_{n}-\tilde{u}_{0}|_{p+1}\leq C\e+C_{\e}C|\tilde{u}_{n}-\tilde{u}_{0}|_{p+1}
\end{align*}
Hence
\begin{align*}
\lim_{n\rightarrow \infty} \left|\int_{\R^{N}} f(\tilde{u}_{n})(\tilde{u}_{n}-\tilde{u}_{0}) \, dx\right|\leq C\e
\end{align*}
and by the arbitrariness of $\e$, we deduce that 
$$
\int_{\R^{N}} f(\tilde{u}_{n})(\tilde{u}_{n}-\tilde{u}_{0}) \, dx\rightarrow 0.
$$
Putting together $(f1)$-$(f3)$ and by using Lemma \ref{Strauss}, we can obtain that
$$
\int_{\R^{N}} f(\tilde{u}_{n})\tilde{u}_{n} \, dx\rightarrow \int_{\R^{N}} f(\tilde{u}_{0})\tilde{u}_{0} \, dx.
$$
Then we can infer that 
$$
\int_{\R^{N}} (f(\tilde{u}_{n})-f(\tilde{u}_{0}))\tilde{u}_{0} \, dx=\int_{\R^{N}}(f(\tilde{u}_{n})\tilde{u}_{n}-f(\tilde{u}_{0})\tilde{u}_{0})\, dx -\int_{\R^{N}}f(\tilde{u}_{n})(\tilde{u}_{n}-\tilde{u}_{0}) \, dx\rightarrow 0.
$$
On the other hand, $\tilde{u}_{n}\rightharpoonup \tilde{u}_{0}$ in $L^{2}(\R^{N})$, so by using the fact that $h\in L^{2}(\R^{N})$, we also have
$$
\int_{\R^{N}} h(x)\tilde{u}_{n} \, dx\rightarrow \int_{\R^{N}} h(x)\tilde{u}_{0} \, dx.
$$
Then, by combining $\langle I'(\tilde{U}_{n}), \tilde{U}_{n} \rangle\rightarrow 0$, $\langle I'(\tilde{U}_{n}), \tilde{U}_{0}\rangle \rightarrow 0$, and the above relations, it follows that $\tilde{U}_{n}\rightarrow \tilde{U}_{0}$ strongly in $\Y$. \\
Therefore, we get 
$$
I(\tilde{U}_{0})= \tilde{c}_{0}<0 \mbox{ and } I'(\tilde{U}_{0})=0.
$$ 

\end{proof}

\noindent
Now, in order to prove that \eqref{3.1} has a mountain pass type solution, we use the following abstract result due to Jeanjean \cite{J}:
\begin{thm}\cite{J}\label{thm3.2}
Let $(X, \|\cdot\|)$ be a Banach space and $J \subset \R_{+}$ be an interval.
Let $(\I_{\lambda})_{\lambda\in J}$ be a family of $C^{1}$ functionals on $X$ of the form
$$
\I_{\lambda}(u)=A(u)-\lambda B(u), \quad \mbox{ for } \lambda \in J,
$$
where $B(u)\geq 0$ for all $u\in X$, and either $A(u)\rightarrow \infty$ or $B(u)\rightarrow \infty$ as $\|u\|\rightarrow \infty$. \\
We assume that there exist $v_{1}, v_{2}\in X$ such that
$$
c_{\lambda}=\inf_{\gamma \in \Gamma} \max_{t\in [0, 1]} \I_{\lambda}(\gamma(t))>\max\{\I_{\lambda}(v_{1}), \I_{\lambda}(v_{2})\}, \quad \forall \lambda \in J
$$
where 
$$
\Gamma=\{\gamma\in C([0, 1], X): \gamma(0)=v_{1}, \gamma(1)=v_{2}\}.
$$
Then, for almost every $\lambda\in J$, there is a sequence $(v_{n})\subset X$ such that
\begin{compactenum}[(i)]
\item $(v_{n})$ is bounded;
\item $\I_{\lambda}(v_{n})\rightarrow c_{\lambda}$;
\item $\I_{\lambda}(v_{n})\rightarrow 0$ on $X^{-1}$.
\end{compactenum}
Moreover, the map $\lambda \mapsto c_{\lambda}$ is continuous from the left hand-side.
\end{thm}

\noindent
For any $\lambda\in [\frac{1}{2}, 1]$, we introduce the following family of functionals $I_{\lambda}: \Y\rightarrow \R$ defined by 
\begin{equation*}
I_{\lambda}(U)= \frac{1}{2}\left[\iint_{\R^{N+1}_{+}} y^{1-2s} |\nabla U|^{2}\, dx dy + \int_{\R^{N}}u^{2} \, dx\right] - \lambda \int_{\R^{N}} (F(u) + h(x)u)\, dx. 
\end{equation*}
for any $U\in \Y$.\\

Next, our claim is to show that $I_{\lambda}$ verifies the assumptions of Theorem \ref{thm3.2}.
\begin{lem}\label{lem3.1}
Assume that $(f1)$-$(f4)$ with $l=+\infty$ hold. Then, 
\begin{compactenum}[$(i)$]
\item There exists $\bar{V}\in \Y\setminus \{0\}$ such that $I_{\lambda}(\bar{V})<0$ for all $\lambda \in [\frac{1}{2}, 1]$.  
\item For $m_{1}>0$ given in Theorem $\ref{thm3.1}$, if $|h|_{2}<m_{1}$, then
\begin{equation*}
c_{\lambda}= \inf_{\gamma \in \Gamma} \max_{t\in [0,1]} I_{\lambda}(\gamma(t))>\max \{I_{\lambda}(0), I_{\lambda}(\bar{V})\}\quad \forall \lambda \in \left[\frac{1}{2}, 1\right],
\end{equation*}
where $\Gamma= \{\gamma \in C([0,1], \Y)) : \gamma(0)=0, \gamma(1)= \bar{V}\}$. 
\end{compactenum}
\end{lem}

\begin{proof}
$(i)$ For any $\delta>0$, we can find $V\in \Y\setminus \{0\}$ and $V\geq 0$ such that 
\begin{equation*}
\iint_{\R^{N+1}_{+}} y^{1-2s}|\nabla V|^{2}\, dx < \delta \int_{\R^{N}} v^{2} dx. 
\end{equation*}
This is lawful due to
$$
\inf\left\{ \iint_{\R^{N+1}_{+}} y^{1-2s}|\nabla U|^{2}\, dx : U\in \Y \mbox{ and } |u|_{2}=1\right\}=0
$$ 
(via the Pohozaev identity, one can see that $(-\Delta)^{s}$ has no eigenvalues in $H^{s}(\R^{N})$).
By using $(f4)$ with $l=+\infty$, and by applying Fatou's lemma, we can deduce that
\begin{equation*}
\lim_{t\rightarrow +\infty} \int_{\R^{N}} \frac{F(tv)}{t^{2}}\, dx \geq (1+\delta) \int_{\R^{N}} v^{2}\, dx. 
\end{equation*}
Hence, for any $\lambda\in [\frac{1}{2}, 1]$, we get
\begin{equation*}
\lim_{t\rightarrow +\infty} \frac{I_{\lambda}(tV)}{t^{2}}\leq \lim_{t\rightarrow +\infty} \frac{I_{\frac{1}{2}}(tV)}{t^{2}}\leq \frac{1}{2} \left( \iint_{\R^{N+1}_{+}} y^{1-2s}|\nabla V|^{2} \, dx dy -\delta \int_{\R^{N}}|v|^{2}\, dx \right) <0. 
\end{equation*}
Take $t_{1}>0$ large enough such that $I_{\frac{1}{2}}(t_{1}V)<0$, and we set $\bar{V}= t_{1}V$. Then, we can see that $I_{\lambda}(\bar{V})\leq I_{\frac{1}{2}}(\bar{V})<0$, that is the condition $(i)$ is satisfied. \\
$(ii)$ It is clear that, for any $\lambda\in [\frac{1}{2},1]$ and $U\in \Y$, we have
\begin{equation*}
I_{\lambda}(U)\geq \frac{1}{2} \left[\int_{\R^{N}} y^{1-2s} |\nabla U|^{2} \, dx dy + \int_{\R^{N}} u^{2} \, dx\right] - \int_{\R^{N}} F(u) \, dx - |h|_{2} |u|_{2} =: J(U). 
\end{equation*}
Then, we can proceed as in the proof of Lemma \ref{lem2.2}, to deduce that
$$
\inf_{\gamma \in \Gamma} \max_{t\in [0,1]} J(\gamma(t))>0,
$$
provided that $|h|_{2}<m_{1}$, with $m_{1}$ given by Theorem \ref{thm3.1}.
Then, for any $\lambda\in \left[\frac{1}{2}, 1\right]$, we can see that
\begin{equation*}
c_{\lambda}= \inf_{\gamma \in \Gamma} \max_{t\in [0,1]} I_{\lambda}(\gamma(t)) \geq \inf_{\gamma \in \Gamma} \max_{t\in [0,1]} J(\gamma(t)) >\max \{I_{\lambda}(0), I_{\lambda}(\bar{V})\}.
\end{equation*}
This ends the proof of the lemma. 

\end{proof}

\noindent
By using Lemma \ref{lem3.1} and Theorem \ref{thm3.2}, we can infer that there exists $\{\lambda_{j}\}\subset [\frac{1}{2}, 1]$ such that
\begin{compactenum}[$(i)$]
\item $\lambda_{j}\rightarrow 1$ as $j\rightarrow +\infty$;
\item $I_{\lambda_{j}}$ has a bounded (PS) sequence $\{U_{n}^{j}\}$ at the level $c_{\lambda_{j}}$. 
\end{compactenum}
In view of Theorem \ref{Lions}, we deduce that for each $j\in \N$, there exists $U_{j} \in \Y$ such that $U_{n}^{j}\rightarrow U_{n}$ strongly in $\Y$ and $U_{j}$ is a positive solution of 
\begin{equation*}
\left\{
\begin{array}{ll}
\dive(y^{1-2s} \nabla U_{j})=0 & \mbox{ in } \R^{N+1}_{+} \\
\frac{\partial U_{j}}{\partial \nu^{1-2s}}=-u_{j}+ \lambda_{j}[f(u_{j})+h(x)]  & \mbox{ on } \R^{N} \\
\end{array}.
\right.
\end{equation*}
Proceeding as in \cite{A1, CW}, it is easy to see that each $U_{j}$ satisfies the following Pohozaev identity: 
\begin{equation}\label{3.2}
\frac{N-2s}{2} \iint_{\R^{N+1}_{+}} y^{1-2s}|\nabla U_{j}|^{2} \, dx dy+ \frac{N}{2} \int_{\R^{N}} u_{j}^{2} \, dx= N \lambda_{j} \int_{\R^{N}} (F(u_{j}) + hu_{j}) \, dx+ \lambda_{j} \int_{\R^{N}} \nabla h(x) \cdot x u_{j} \, dx. 
\end{equation}

\noindent 
In the next lemma, we use the condition $(H)$ to prove the boundedness of the sequence $\{U_{j}\}$.
\begin{lem}\label{lem3.2}
Assume that $(f1)$-$(f4)$ with $l=+\infty$ hold, and $h$ satisfies (1.5) and $|h|_{2}<m_{1}$ for $m_{1}$ given in Theorem \ref{thm3.1}. Then $\{U_{j}\}\subset \Y$ is bounded. 
\end{lem}

\begin{proof}
By using Theorem \ref{thm3.2}, we know that the map $\lambda\rightarrow c_{\lambda}$ is continuous from the left.
Then, by Lemma \ref{lem3.1} $(ii)$, we can deduce that $I_{\lambda_{j}}(U_{j})= c_{\lambda_{j}}\rightarrow c_{1}>0$ as $\lambda_{j}\rightarrow 1$. 

Hence, we can find a constant $K>0$ such that $I_{\lambda_{j}}(U_{j})\leq K$ for all $j\in \N$. 
By combining this, \eqref{3.2}, $u_{j}>0$ and $(H)$,  we can see that
\begin{equation*}
\iint_{\R^{N+1}_{+}} y^{1-2s} |\nabla U_{j}|^{2} \, dx dy \leq \frac{KN}{s} - \frac{\lambda_{j}}{s}\int_{\R^{N}} \nabla h(x) \cdot x u_{j}\, dx\leq \frac{KN}{s},
\end{equation*}
which together with the Sobolev inequality (\ref{Mia}), implies that
\begin{equation}\label{AH10}
|u_{j}|_{2^{*}_{s}}\leq S_{*} \left(\iint_{\R^{N+1}_{+}} y^{1-2s} |\nabla U_{j}|^{2} \, dx dy\right)^{\frac{1}{2}}\leq C.
\end{equation}
Now, from $I_{\lambda_{j}}(U_{j})\leq K$ for all $j\in \N$, it follows that
\begin{equation}\label{3.8}
\frac{1}{2}\left[\iint_{\R^{N+1}_{+}} y^{1-2s} |\nabla U_{j}|^{2} \, dx dy + \int_{\R^{N}} u_{j}^{2}\, dx\right] - \lambda_{j} \int_{\R^{N}} (F(u_{j})+ h(x) u_{j})\, dx \leq K. 
\end{equation}
On the other hand, by using $(f2)$, $(f3)$, we can see that there exists a constant $C>0$ such that
\begin{equation*}
\int_{\R^{N}} F(u_{j}) \, dx\leq \frac{1}{4} \int_{\R^{N}} u_{j}^{2} \, dx+ C\int_{\R^{N}} u_{j}^{2^{*}} \, dx, 
\end{equation*}
where $2^{*}_{s}= \frac{2N}{N-2s}$. 

Substituting this inequality into (\ref{3.8}), and by using (\ref{AH10}), (\ref{Mia}), we deduce that
\begin{align*}
\frac{1}{2} \int_{\R^{N}} u_{j}^{2} \, dx &\leq \lambda_{j} \int_{\R^{N}} (F(u_{j})+ h(x) u_{j})\, dx+K \\
&\leq \frac{1}{4}|u_{j}|_{2}^{2}+C|u_{j}|_{2^{*}_{s}}^{2^{*}_{s}}+|h|_{2}|u_{j}|_{2}+K \\
&\leq  \frac{1}{4}|u_{j}|_{2}^{2}+\bar{C}+|h|_{2}|u_{j}|_{2}+K.
\end{align*}
Then
\begin{equation*}
\frac{1}{4} \int_{\R^{N}} u_{j}^{2} \, dx \leq \tilde{C} + |h|_{2} |u_{j}|_{2}, 
\end{equation*}
that is
\begin{equation}\label{l2bound}
|u_{j}|_{2}\leq C \mbox{ for all } j\in \N,
\end{equation}
for some positive constant $C$ independent of $j$.
Putting together (\ref{AH10}) and (\ref{l2bound}), we can conclude the proof of this lemma. 

\end{proof}

\begin{lem}\label{lem3.3}
Under the assumptions of Lemma $\ref{lem3.2}$, the above sequence $\{U_{j}\}$ is also a (PS) sequence of $I$. 
\end{lem}

\begin{proof}
From the definitions of $I$ and $I_{\lambda_{j}}$ we can deduce that
\begin{equation}\label{3.9}
I(U_{j})= I_{\lambda_{j}}(U_{j}) + (\lambda_{j}-1) \int_{\R^{N}} (F(u_{j}) + h(x) u_{j})\, dx. 
\end{equation}
By using Theorem \ref{thm3.2}, we  obtain 
$$
I_{\lambda_{j}}(U_{j})= c_{\lambda_{j}}\rightarrow c_{1}>0 \mbox{ as } \lambda_{j}\rightarrow 1.
$$ 
Hence, by applying Lemma \ref{lem3.2} and \eqref{3.9}, we get $I(U_{j})\rightarrow c_{1}>0$. Being $I'_{\lambda_{j}}(U_{j})=0$, we can infer that, for any $\varPsi \in C^{\infty}_{0}(\R^{N+1}_{+})$, 
\begin{equation*}
\langle I'(U_{j}), \varPsi \rangle = \langle I'_{\lambda_{j}}(U_{j}), \varPsi \rangle + (\lambda_{j}-1)\int_{\R^{N}} (f(u_{j}) +h(x))\psi \, dx \rightarrow 0, 
\end{equation*}
that is $I'(U_{j})\rightarrow 0$  as $j\rightarrow \infty$ in the dual space of $\Y$.

\end{proof}

\noindent

Finally, we give the proof of the main result of this section: 
\begin{proof}[Proof of Theorem 1.2]
Taking into account Theorem \ref{thm3.1}, we know that \eqref{3.1} admits a positive solution $\tilde{U}_{0}\in \Y$ such that $I(\tilde{U}_{0})<0$. On the other hand, by Lemma \ref{lem3.3} and Theorem \ref{Lions}, we know that problem \eqref{3.1} has a second positive solution $\tilde{U}_{1}\in \Y$ with $I(\tilde{U}_{1})= c_{1}>0$. \\
As a consequence $\tilde{U}_{0}\not \equiv \tilde{U}_{1}$, and this ends the proof.

\end{proof}

\section{Some examples}

\noindent
In this last section we provide some examples of functions $f$, $k$ and $h$ for which our main results are applicable.
\begin{ex}
Let $R_{0}>0$ and let us define
\begin{equation*}
k(x)=
\left\{
\begin{array}{ll}
\frac{1}{1+|x|} &\mbox{ if } |x|<R_{0}\\
\frac{1}{1+R_{0}} &\mbox{ if } |x|\geq R_{0} 
\end{array}
\right.
\mbox{ and }
f(t)=
\left\{
\begin{array}{ll}
\frac{R_{0}t^{2}}{1+t} &\mbox{ if } t>0\\
0 &\mbox{ if } t\leq 0. 
\end{array}
\right.
\end{equation*}
It is clear that $|k|_{\infty}=1$, and $f$ satisfies $(f1)$-$(f3)$ and $(f4)$ with $l=R_{0}$. Moreover, we note that (K) holds because of
\begin{equation*}
\sup \left\{\frac{f(t)}{t}: t>0\right\}=R_{0}<R_{0}+1=\inf\left\{\frac{1}{k(x)}: |x|\geq R_{0}\right\}.\\
\end{equation*}
Now, to verify that $l>\mu^{*}$, we have to choose a special $R_{0}>0$. For $R>0$, we take $\phi \in C^{\infty}_{0}(\R^{N})$ such that $\phi(x)=1$ if $|x|\leq R$, $\phi(x)=0$ if $|x|\geq 2R$ and $|\nabla \phi(x)|\leq \frac{C}{R}$ for all $x\in \R^{N}$. 

Since $\phi\in H^{1}(\R^{N})\subset H^{s}(\R^{N})$, we can see that 
$$
\|\phi\|_{H^{s}(\R^{N})}\leq C\|\phi\|_{H^{1}(\R^{N})}.
$$ 
On the other hand, for any $R_{0}>2R$, we have
\begin{equation*}
\frac{\int_{\R^{N}} \phi^{2} dx}{\int_{\R^{N}} k(x)\phi^{2} dx} \leq \frac{\int_{\R^{N}} \phi^{2} dx}{\frac{1}{1+2R}\int_{\R^{N}} \phi^{2} dx}= 1+2R
\end{equation*}
and 
\begin{equation*}
\frac{\int_{\R^{N}} |\nabla \phi|^{2} dx}{\int_{\R^{N}} k(x)\phi^{2} dx}\leq \frac{\frac{C^{2}}{R^{2}} |B_{2R}|}{\int_{B_{R}} k(x)\, dx} \leq \frac{\frac{C^{2}}{R^{2}} |B_{2R}|}{\frac{1}{1+R}|B_{R}|}= C_{1}\frac{(1+R)}{R^{2}}.
\end{equation*}
Therefore 
\begin{equation*}
\frac{\|\phi\|^{2}_{H^{s}(\R^{N})}}{\int_{\R^{N}} k(x)\phi^{2} dx} \leq \frac{C\|\phi\|^{2}_{H^{1}(\R^{N})}}{\int_{\R^{N}} k(x)\phi^{2} dx}\leq C_{2}\frac{(1+R)}{R^{2}}+C_{3}(1+2R)
\end{equation*}
where $C_{2}, C_{3}>0$ are constants independent of $R$. \\
Choosing $R>0$ such that $C_{2}\frac{(1+R)}{R^{2}}\leq C_{3}$, we can infer that $\mu^{*}\leq 2C_{3}(R+1)$. Then, taking $R_{0}= 2C_{3}(R+1)+2R$, we have
\begin{equation*}
\lim_{t\rightarrow +\infty} \frac{f(t)}{t}= l= R_{0}>\mu^{*}. 
\end{equation*} 
Now, fix $\e\in (0, 1)$, and let $h\in L^{2}(\R^{N})$ such that 
\begin{equation*}
|h|_{2}<m:=\max_{t\geq 0}\left[\left(\frac{1}{2}-\frac{\varepsilon}{2}\right)t-\frac{C_{\varepsilon}}{p+1}S_{*}^{p+1}t^{p}\right].
\end{equation*}
Then, all assumptions of Theorem \ref{thm1} are satisfied, and we can find at least two positive solutions to \eqref{P}.
\end{ex}

\begin{ex}
Fix $\e\in (0, 1)$, and let us consider the following functions
\begin{equation*} 
h(x)=
\left\{
\begin{array}{ll}
0 &\mbox{ if } |x|<\sqrt{3} \vee |x|>2\\
C(|x|^{2}-2)^{2}(|x|^{2}-3)^{2}(|x|^{2}-4)^{2} &\mbox{ if } \sqrt{3}\leq |x|\leq 2 
\end{array}
\right.
\end{equation*}
and
\begin{equation*} 
f(t)=
\left\{
\begin{array}{ll}
t \log(1+t) &\mbox{ if } t>0\\
0 &\mbox{ if } t\leq 0, 
\end{array}
\right.
\end{equation*}
where $C>0$ is a constant such that 
$$
|h|_{2}<m:=\max_{t\geq 0}\left[\left(\frac{1}{2}-\frac{\varepsilon}{2}\right)t-\frac{C_{\varepsilon}}{p+1}S_{*}^{p+1}t^{p}\right].
$$
It is clear that $f$ satisfies $(f1)$-$(f3)$ and $(f4)$ with $l=\infty$, and $h\in C^{1}(\R^{N})\cap L^{2}(\R^{N})$. \\
Moreover, for any $\sqrt{3}<|x|<2$, we have 
$$
x\cdot \nabla h= 4C\left[|x|^{2} (|x|^{2}-2)(|x|^{2}-3)(|x|^{2}-4)(3|x|^{4}-22|x|^{2}+26)\right]\geq 0,
$$  
so $x\cdot \nabla h\geq 0$ on $\R^{N}$. In particular, $x\cdot \nabla h\in L^{q}(\R^{N})$ for any $q\in [1, \infty]$.
Then, we can apply Theorem \ref{thm2} to deduce that the problem \eqref{P} admits at least two positive solutions.
\end{ex}

\end{document}